\documentclass[11pt]{article}
\usepackage{amssymb,a4wide,latexsym,makeidx,epsfig}
\usepackage{amsthm}
\usepackage{amsmath}
\usepackage{cite}
\usepackage{graphicx}
\usepackage{float}
\usepackage{enumerate}
\usepackage{booktabs}
\usepackage{arydshln}
\usepackage{listings}
\usepackage[usenames,dvipsnames]{xcolor}
\usepackage{wrapfig}
\usepackage{booktabs}
\usepackage{verbatim} 

\usepackage{setspace}
\usepackage{geometry}
\newtheorem{theorem}{Theorem}
\newtheorem{lemma}{Lemma}
\newtheorem{fact}{Fact}

\newtheorem{conjecture}{Conjecture}
\newtheorem{claim}{Claim}
\newtheorem{problem}{Problem}

\newtheorem{case0}{Case}
\newtheorem{subcase0}{Subcase}[case0]
\newtheorem{case4}{Case}

\newtheorem{case5}{Case}

\newtheorem{case10}{Case}

\begin{document}

\title{Vertex-disjoint cycles of different lengths in tournaments
\thanks{Supported by National Natural Science Foundation of China (Grant Nos. 12311540140, 12242111, 12131013, 11601430), Guangdong Basic \& Applied Basic Research Foundation (Grant Nos. 2023A1515030208, 2022A1515010899), Shaanxi Fundamental Science Research Project for Mathematics and Physics (Grant No. 22JSZ009), and the Fundamental Research Funds for the Central Universities (Grant No. G2023KY0606).}}

\author{\quad Yandong Bai $^{a, b}$,  
\thanks{Corresponding author. 
E-mail addresses: 
bai@nwpu.edu.cn (Y. Bai),
wenpei.jia@foxmail.com (W. Jia).}
\quad Wenpei Jia  $^{a, b}$ \\
\small $^{a}$ School of Mathematics and Statistics, Northwestern Polytechnical University, \\
\small Xi'an, Shaanxi 710129, China\\
\small $^{b}$ Research $\&$ Development Institute of 
Northwestern Polytechnical University in Shenzhen, \\
\small Shenzhen, Guangdong 518057, China}
\date{}
\maketitle

\begin{abstract}
Bermond and Thomassen conjectured in 1981 that 
every digraph with minimum outdegree at least $2k-1$ 
contains $k$ vertex-disjoint cycles,
here $k$ is a positive integer.
Lichiardopol conjectured in 2014 that for every positive integer $k$ 
there exists an integer $g(k)$ such that 
every digraph with minimum outdegree at least $g(k)$ 
contains $k$ vertex-disjoint cycles of different lengths.
Recently, 
Chen and Chang proved in [J. Graph Theory 105 (2) (2024) 297-314] 
that for $k\geqslant 3$ every tournament with minimum outdegree at least $2k-1$ contains $k$ vertex-disjoint cycles in which two of them have different lengths.
Motivated by the above two conjectures and related results,
we investigate vertex-disjoint cycles of different lengths in tournaments,
and show that when $k\geqslant 5$ every tournament with minimum outdegree at least $2k-1$ contains $k$ vertex-disjoint cycles in which three of them have different lengths.
In addition,
we show that every tournament with minimum outdegree at least $6$ contains three vertex-disjoint cycles of different lengths and the minimum outdegree condition is sharp.
This answers a question proposed by Chen and Chang.
\medskip

\end{abstract}

\section{Introduction}

Throughout this paper, 
a cycle (path) in a digraph always means a \textit{directed} cycle (path). 
We use Bang-Jensen and Gutin \cite{BG} for terminology and
notation not defined here. Only finite and simple digraphs are considered.

For a positive integer $k$, 
let $f(k)$ be the minimum integer such that 
every digraph with minimum outdegree at least $f(k)$ 
contains $k$ vertex-disjoint cycles,
and let $g(k)$ be the minimum integer such that 
every digraph with minimum outdegree at least $g(k)$ 
contains $k$ vertex-disjoint cycles of different lengths.
Clearly, $f(k)\leqslant g(k)$.

The finiteness of the above two functions are not obvious.
Thomassen \cite{T} was the first who obtained the finiteness result of $f(k)$,
to be precise, it had been proved that $f(k)\leqslant (k+1)!$.
In view of the complete symmetric digraphs,
one can see that $f(k)\geqslant 2k-1$.
In 1981, 
Bermond and Thomassen \cite{BT} conjectured that the equality holds,
which is selected as one of the hundred open problems listed in the monograph by Bondy and Murty \cite{BM2008}.

\begin{conjecture}[Bermond and Thomassen \cite{BT}]\label{conj:BTC}
$f(k)=2k-1$.
\end{conjecture}

In 1996, via probabilistic arguments,
Alon \cite{A} showed that $f(k)\leqslant 64k$ and obtained the first linear upper bound of $f(k)$.
Buci\'{c} \cite{B} reduced the bound to $18k$ in 2018.
In addition,
Conjecture \ref{conj:BTC} clearly holds for $k=1$.
Thomassen \cite{T} in 1983 and Lichiardopol et al. \cite{LPS} in 2009 
proved the conjecture for $k=2$ and $k=3$, respectively.
The first author of this paper and Manoussakis \cite{BM} gave a much shorter proof for the case of $k=3$ in 2019.
It is widely open for $k\geqslant 4$.

For the function $g(k)$,
one can see from the definition that $g(1)=f(1)=1$.
Lichiardopol \cite{L} showed that $g(2)=4$.
For $k\geqslant 3$,
Lichiardopol \cite{L} conjectured in 2014 that $g(k)$ is finite but
no (finite) upper bound has been obtained for any $k\geqslant 3$ till now.

\begin{conjecture}[Lichiardopol \cite{L}]\label{conj:LC}
$g(k)$ is finite for $k\geqslant 3$.
\end{conjecture}

We refer the reader to \cite{BLL,BJHA,HY,L,S,T2,T3,WQY} 
for more results on Conjectures \ref{conj:BTC} and \ref{conj:LC}. 

In this paper,
we concentrate on vertex-disjoint cycles of different lengths
and investigate the following more general function in digraphs.
For two positive integers $k,\ell$ with $k\geqslant \ell$, 
define $h(k,\ell)$ to be the minimum integer such that 
every digraph with minimum outdegree at least $h(k,\ell)$ 
contains $k$ vertex-disjoint cycles in which $\ell$ of them have different lengths.
It is not difficult to check that the following holds. 

\begin{fact}
$f(k)=h(k,1)$ and $g(k)=h(k,k)$.
\end{fact}

In other words,
Conjectures \ref{conj:BTC} and \ref{conj:LC} can be reconsidered as 
establishment problems of two special values of $h(k,\ell)$.
Here we need to mention that this function is motivated 
by the problem investigated by Chen and Chang in \cite{CC}.

A \textit{tournament} is a digraph satisfying that 
there exists exactly one arc between every two distinct vertices,
or equivalently,
a tournament is an orientation of a complete graph.
The class of tournaments plays a very important role in digraph theory.
For convenience,
we denote the corresponding $h(k,\ell)$ in tournaments 
by $h^*(k,\ell)$ in the rest of this paper.

Bang-Jensen et al. \cite{BBT} proved that $h^*(k,1)\leqslant 2k-1$ in 2014.
Bensmail et al. \cite{BJHA} proved that $\frac{k^2+5k-2}{4}\leqslant h^*(k,k)\leqslant \frac{k^2+4k-3}{2}$ in 2017.
Tan \cite{T5} showed in 2021 that a strongly connected tournament 
with minimum outdegree at least 3
contains no two disjoint cycles of different lengths 
if and only if it is isomorphic to a specific tournament with minimum outdegree 3. 
It therefore implies that $h^*(2,2)=4$.
Recently,
Chen and Chang \cite{CC} showed that $h^*(k,2)\leqslant 2k-1$ for $k\geqslant 3$.
In this paper, 
we show that $h^*(k,3)\leqslant 2k-1$ for $k\geqslant 5$
and, by answering a question of Chen and Chang, $h^*(3,3)=6$.

\begin{theorem}\label{thm:main}
Every tournament with minimum outdegree at least $2k-1$ contains $k$ disjoint cycles in which three of them have different lengths, here $k\geqslant 5$ is an integer.
\end{theorem}

\begin{theorem}\label{thm:main2}
Every tournament with minimum outdegree at least $6$ contains three disjoint cycles with different lengths and the minimum outdegree condition is sharp.
\end{theorem}

The rest of this paper is organized as follows. 
In Section \ref{section:pre} we introduce some necessary terminology and notations.
Sections \ref{section:proof} and \ref{section:proof1}
are devoted to the proofs of Theorems \ref{thm:main} and \ref{thm:main2}, respectively.
The final section concludes with some remarks and open problems.

\section{Additional terminology and notation}
\label{section:pre}

If $a=(u,v)$ is an arc of a digraph $D$, 
then we say that $u$ \emph{dominates} $v$ and denote it by $u\rightarrow v$.
For two vertex subsets $X,Y$ of $D$, 
we denote the set of arcs from $X$ to $Y$ by $A(X,Y)$,
and write $X\rightarrow Y$ if every vertex of $X$ dominates all vertices of $Y$.
The vertices which dominate a vertex $v$ are its \emph{in-neighbors}, 
those which are dominated by $v$ are its \emph{out-neighbors},
which are denoted by $N_D^-(v)$ and $N_D^+(v)$.
The \emph{indegree} (\emph{outdegree}) is the cardinality of the $N_D^-(v)$ ($N_D^+(v)$), denoted by $d^-_D(v)$ ($d^+_D(v)$), respectively.
The minimum indegree and outdegree of $D$ are denoted by $\delta ^-(D)$ and $\delta ^+(D)$, respectively; the maximum indegree and outdegree of $D$ are denoted by $\Delta^-(D)$ and $\Delta^+(D)$, respectively.

A digraph $D$ is \emph{$r$-diregular} 
if  $d^-_D(v)=d^+_D(v)=r$ for every vertex $v\in V(D)$.
For a set $W\subseteq V$, $N^+_D(W)$ ($N^-_D(W)$) consists of those vertices from $V-W$ which are out-neighbors (in-neighbors) of at least one vertex from $W$.
A digraph $H$ is a \emph{subdigraph} of a digraph $D$ if $V(H)\subseteq V(D)$, $A(H)\subseteq A(D)$ and every arc in $A(H)$ has both end-vertices in $V(H)$.
If every arc of $A(D)$ with both end-vertices in $V(H)$ is in $A(H)$, 
then we say that $H$ is induced by $X=V(H)$ ($H=D[X]$) 
and call $H$ an \emph{induced subdigraph} of $D$.
A digraph $D$ is \emph{acyclic} if it has no cycle.
If $H$ is an induced subdigraph of tournament $T$, 
we say that $H$ is a \emph{subtournament} of $T$.
In a tournament $T$, if $u\rightarrow v$ and $v\rightarrow w$ can imply that $u\rightarrow w$, 
then $T$ is a \emph{transitive tournament}. 
Note that the acyclic tournament is the unique transitive tournament.

In a digraph $D$, 
a vertex $x$ is \emph{reachable} to a vertex $y$ 
if there exists an $(x, y)$-path in $D$.
In particular, a vertex is reachable to itself.
A digraph $D$ is \emph{strongly connected} if 
there exists an $(x, y)$-path
for every two vertices $x$, $y$ in $D$.

A \emph{strong component} of a digraph $D$ 
is a strongly connected induced subdigraph of $D$ with maximal number of vertices. 
The \emph{strong component digraph} $SC(D)$ of $D$ is obtained by contracting all the strong components of $D$ and deleting every parallel arcs obtained in this process.
In other words, 
if $D_1$,\ldots,$D_t$ are all the strong components of $D$, 
then 
\[
V(SC(D))=\{v_1, \ldots, v_t\},~
A(SC(D))=\{v_iv_j : A(V(D_i), V(D_j))\neq\emptyset\}.
\]
Note that the subdigraph induced by the vertices of a cycle in $D$ is strongly connected.
Thus $SC(D)$ is acyclic.
It is not difficult to check that 
the strong component digraph of a tournament is a transitive tournament.

A digraph $D$ is \emph{vertex-pancyclic} if, 
for every vertex $v$ of $D$ and 
for every positive integer $3\leqslant \ell\leqslant n$, 
there is a cycle of length $\ell$ passing through $v$.
In 1966, Moon \cite{M} showed the following classical result on tournaments, 
which will play an important role in our proofs.

\begin{lemma}[Moon \cite{M}]\label{thm:Moon}
Every strongly connected tournament is vertex-pancyclic.
\end{lemma}

To simplify the proof process,
we define a new useful concept.
Let $r$ be a positive integer.
Call a digraph \textit{$r$-outdegree-critical} 
if it has minimum oudegree $r$ and has minimal number of vertices, i.e., 
the deletion of any vertex decreases the minimum outdegree.
The following lemma will be used several times in our proofs.

\begin{lemma}\label{th:th2.4}
Every $r$-outdegree-critical tournament is strongly connected.
\end{lemma}

\begin{proof}
Let $T$ be an $r$-outdegree-critical tournament
and let $v_1, v_2$ be two arbitrary distinct vertices of $T$.
Assume without loss of generality that $v_1\rightarrow v_2$,
it suffices to show that $v_2$ is reachable to $v_1$.
If there is a vertex $w\in N_T^{+}(v_2)$ with $w\rightarrow v_1$,
then $v_2$ is reachable to $v_1$ through the path $P_1=(v_2,w,v_1)$.
Otherwise, $v_1\rightarrow \{v_2\}\cup N_T^{+}(v_2)$.
By the definition of outdegree-critical,
every vertex is dominated by some vertex having minimum outdegree.
Thus there exists a vertex $u\notin N_T^+(v_2)$ 
with $u\rightarrow v_1$ and $d^{+}_{T}(u)=r$.
It follows that $u$ has at most $r-1$ outneighbors in $N_T^{+}(v_2)$.
Since $|N_T^{+}(v_2)|\geqslant r$, 
there exists a vertex $w'\in N_T^{+}(v_2)$ with $w'\rightarrow u$.
Now $v_2$ is reachable to $v_1$ through the path $P_2=(v_2,w',u,v_1)$.
\end{proof}

\section{Proof of Theorem \ref{thm:main}}
\label{section:proof}

We shall present the proof by contradiction.
For convenience,
call a collection of $k$ vertex-disjoint cycles \emph{good} 
if three of them have different lengths.
Suppose that $T_n$ is a $(2k-1)$-outdegree-critical tournament on $n$ vertices but contains no good collection of vertex-disjoint cycles.
Chen and Chang \cite{CC} showed that $h^*(k,2)\leqslant 2k-1$.
Thus $T_n$ contains a collection of $k$ vertex-disjoint cycles, say $\mathcal{C}=\{C_1, \ldots, C_k\}$, 
in which two of them have different lengths.
Without loss of generality, assume that 
$|V(C_1)|=\cdots=|V(C_{k-1})|=3$ and $|V(C_k)|=4$.
Denote by $C_{k}=(z_1,z_2,z_3,z_4,z_1)$
and $C_i=(y_0^i,y_1^i,y_2^i,y_0^i)$ for $1\leqslant i\leqslant k-1$, .
Up to isomorphism,
assume that $z_1\rightarrow z_3$ and $z_2\rightarrow z_4$.
Let $V(\mathcal{C})=V(C_1)\cup \ldots\cup V(C_{k})$, $X=V(T_n)\backslash V(\mathcal{C})$, $Y=V(\mathcal{C})\backslash V(C_k)$ and $Z=V(C_k)$.
It is not difficult to see that
$\lvert V(\mathcal{C})\rvert=3k+1$
and $\lvert X\rvert=n-3k-1$,
denote $\lvert X\rvert$ by $t$.

Let $P=x_1, x_2, \ldots, x_t$ be a Hamiltonian path in $T[X]$, where $x_t\in V(X_s)$ is a vertex with the minimum outdegree in $X_s$.
The desired contradiction will appear after eight claims.

\begin{claim}\label{cl:cl1}
For every two distinct vertices $x_p,x_q\in X$ with $p<q$
and for every $1\leqslant i\leqslant k-1$, 
if $\lvert A(V(C_i),x_p)\rvert \geqslant 1$, 
then $V(C_i)\rightarrow x_q$;
if $\lvert A(x_q,V(C_i))\rvert \geqslant 1$, 
then $x_p\rightarrow V(C_i)$.
\end{claim}

\begin{proof}
Suppose that $\lvert A(V(C_i),x_p)\rvert \geqslant 1$ but $\lvert A(x_q,V(C_i))\rvert \geqslant 1$ for some $1\leqslant i\leqslant k-1$.
This implies that $T[V(C_i)\cup\{x_p,\ldots,x_q\}]$ is a strong subtournament of $T_n$ with at least $5$ vertices,
which contains a $5$-cycle $C_i^{*}$ by Lemma \ref{thm:Moon}.
It follows that $\mathcal{C}^*=\{C_1,\ldots,C_{i-1},C_i^{*},C_{i+1},\ldots,C_k\}$ is a good collection, a contradiction.
Similarly, we can show that for every $1\leqslant i\leqslant k-1$
if $\lvert A(x_q,V(C_i))\rvert \geqslant 1$ then $x_p\rightarrow V(C_i)$.
\end{proof}

Now we consider the strong component digraph $SC(T[X])$,
which is obtained by contracting all the strong components $X_1, X_2, \ldots, X_r$ of $X$.

\begin{claim}\label{cl:cl2}
$T[X]$ is acyclic.
\end{claim}

\begin{proof}
Suppose that $T[X]$ is not acyclic.
It follows that $\lvert V(X_r)\rvert=1$, $3$ or $4$, 
here $1\leqslant r\leqslant t-2$.
We shall finish the proof of this claim through the following three statements.

\begin{itemize}
\item For every two distinct cycles $C_i, C_j$ with $1\leqslant i,j\leqslant k-1$, either $V(C_i)\rightarrow V(C_j)$ or  $V(C_j)\rightarrow V(C_i)$.
\end{itemize}

Suppose that $A(V(C_i),V(C_{j_1}))\neq \emptyset$ and $A(V(C_{j_1}),V(C_i))\neq \emptyset$.
Then $T[V(C_i)\cup V(C_{j_1})]$ is a strong subtournament of $T_n$ with $6$ vertices,
which contains a $5$-cycle $C_i^{*}$.
By assumption, 
$T[X]$ contains a $3$-cycle $C^{*}$.
Now
\[
\mathcal{C}^*=\{C_1,\ldots,C_{i-1},C_i^{*},C_{i+1},\ldots,C_{j_1-1},C_{j_1+1},\ldots,C_k,C^*\}
\]
is a good collection, a contradiction.

\begin{itemize}
\item We can reorder the cycles in $Y$ such that $V(C_i)\rightarrow V(C_j)$ for any 
$1\leqslant i<j\leqslant k-1$.
\end{itemize}

Suppose that $V(C_i)\rightarrow V(C_{j_1})$ but $V(C_{j_2})\rightarrow V(C_i)$ for some $1\leqslant i<j_1<j_2\leqslant k-1$.
Let $C_i^*=(y_0^i, y_0^{j_1}, y_0^{j_2},y_0^i)$
and $C_{j_1}^*=(y_1^i, y_2^i, y_1^{j_1}, y_2^{j_1}, y_2^{j_2}, y_1^i)$.
Note that $T[X]$ contains a $3$-cycle $C^{*}$.
Now
\[
\mathcal{C}^*
=\{C_1,\ldots,C_{i-1},C_i^{*},C_{i+1},\ldots,C_{j_1-1},C_{j_1}^*,C_{j_1+1},\ldots,C_{j_2-1},C_{j_2+1},\ldots,C_k,C^*\}
\]
is a good collection, a contradiction.

Without loss of generality, we may assume that for every three distinct cycles $C_i$, $C_{j_1}$ and $C_{j_2}$, it always have $V(C_i)\rightarrow V(C_{j_1})$ and $V(C_i)\rightarrow V(C_{j_2})$ for $1\leqslant i<j_1<j_2\leqslant k-1$.
Clearly,
\begin{align*}
\begin{split}
d^+_{T[X]}(x_t)=d^+_{X_s}(x_t)=\left\{
\begin{array}{lr}
0, &\lvert V(X_s)\rvert=1;\\
1, &\lvert V(X_s)\rvert=3;\\
1, &\lvert V(X_s)\rvert=4.
\end{array}
\right.
\end{split}
\end{align*}

\begin{itemize}
\item $V(C_{k-1})\rightarrow V(X_s)$.
\end{itemize}

We can show that $V(C_{k-1})\rightarrow x_t$.
If not, then we suppose that there exists $0\leqslant i\leqslant 2$
such that $x_t\rightarrow y_i^{k-1}$.
Recall that $V(C_j)\rightarrow V(C_{k-1})$ for every $1\leqslant j\leqslant k-1$.
So we have $d^+_{T[V(\mathcal{C})]}(y_i^{k-1})\leqslant 5$.
Since $d^+_{T_n}(y_i^{k-1})\geqslant 2k-1\geqslant 9$,
we have $y_i^{k-1}\rightarrow x_p$ for some $1\leqslant p\leqslant t-1$,
a contradiction to Claim \ref{cl:cl1}.
If $\lvert V(X_s)\rvert=1$, then the statement is proved.
If $\lvert V(X_s)\rvert=3$ or $4$,
then suppose that $A(x_q,V(C_{k-1}))\neq \emptyset$ for some $x_q\in V(X_s)$.
Thus $T[V(C_{k-1})]\cup V(X_s)$ is a strong subtournament of $T_n$ with at least $6$ vertices,
which contains a $5$-cycle $C_{k-1}^*$.
Now
$
\mathcal{C}^*=\{C_1,\ldots,C_{k-1}^*,C_k\}   
$
is a good collection, a contradiction.

Recall that $\lvert V(X_s)\rvert=1$, $3$ or $4$.
Suppose that $\lvert V(X_s)\rvert=1$.
Since $d^+_{T[X]}(x_t)=0$ and $d^+_{T_n}(x_t)\geqslant 2k-1$, then there exists a cycle $C_i$ such that $d^+_{T[V(C_i)]}(x_t)\geqslant 1$, where $1\leqslant i\leqslant k-2$.
This implies that $T[V(C_i)\cup V(C_{k-1})\cup\{x_t\}]$ is a strong subtournament of $T_n$ with at least $7$ vertices,
which contains a $5$-cycle $C_i^{*}$ by Lemma \ref{thm:Moon}.
Note that $T[X]$ contains a $3$-cycle $C^{*}$.
Now
\[
\mathcal{C}^*
=\{C_1,\ldots,C_{i-1},C_{i}^*,C_{i+1},\ldots,C_{k-2},C_k,C^*\}
\]
is a good collection, a contradiction.

Then $\lvert V(X_s)\rvert=3$ or $4$.
Since $d^+_{T[X]}(x_t)=1$ and $d^+_{T_n}(x_t)\geqslant 2k-1$, then $d^+_{T[V(C_i)]}(x_t)\geqslant 1$, where $1\leqslant i\leqslant k-2$.
It follows that $x_j\rightarrow V(C_i)$ by Claim \ref{cl:cl1}, here $1\leqslant j\leqslant t-1$.
Let $C^{*}_i=(x_{t-2},y_0^i,y_1^i,y_0^{k-1},y_1^{k-1},x_{t-2})$,
$C^{*}_{k-1}=(x_{t-1},y_2^i,y_2^{k-1},x_{t-1})$.
Then we can get a good collection $\mathcal{C}^*$ by replacing $C_i$, $C_{k-1}$ with $C_i^*$, $C_{k-1}^*$, a contradiction.
\end{proof}

It follows that $\lvert V(X_i)\rvert=1$ and $s=t$.
Let $V(X_i)=\{x_i\}$.
Suppose that $\lvert A(x_t,V(C_{j_1}))\rvert\geqslant \lvert A(x_t,V(C_{j_2}))\rvert$ and $\lvert A(x_t,V(C_1))\rvert$ is as large as possible, where $1\leqslant j_1<j_2\leqslant k-1$.
Because of the choice of $C_1$ and Claim \ref{cl:cl1},
we can conclude that $\lvert A(x_t,V(C_1))\rvert\geqslant 2$.
If not, then assume that $x_t\rightarrow y_0^1$ and $\{y_1^1,y_2^1\}\rightarrow x_t$.
Let $C^{'}_1=(y_0^1,y_1^1,x_t,y_0^1)$ and $v_t^{'}=y_2^1$.
Then $\lvert A(x_t^{'},V(C_1^{'}))\rvert=2$,
a contradiction to the choice of $C_1$.

Let $Y_1=V(C_1)\cup\cdots\cup V(C_h)$ satisfying that $A(x_t,V(C_j))\neq\emptyset$ for every $1\leqslant j\leqslant h$.
Denote $Y_2=V(C_{h+1})\cup\cdots\cup V(C_{k-1})$.
One can check that $Y_2\rightarrow v_t$.

\begin{claim}\label{cl:cl3}
If there exists a cycle $C_i$
such that $1 \leqslant\lvert A(x_t,V(C_i))\rvert\leqslant 2$,
then $N^{+}_{T[X\backslash \{x_t\}]}(z_{j_1})=N^{+}_{T[X\backslash \{x_t\}]}(z_{j_2})$,
where $1\leqslant i\leqslant h$,
$j_1\neq j_2$ and $j_1,j_2\in \{1,2,3,4\}$;
moreover, for every two distinct vertices $x_p,x_q$,
if $x_p\in N^+_{T_n}(V(C_k))$,
then $x_q\in N^+_{T_n}(V(C_k))$,
where $1\leqslant p<q\leqslant t-1$.
\end{claim}

\begin{proof}
If not, then there exists $1\leqslant p\leqslant t-1$
such that $z_{j_1}\rightarrow x_p$ and $x_p\rightarrow z_{j_2}$.
It follows that $T[V(C_k)\cup \{x_p\}]$ is a strong subtournament of $T_n$ with $5$ vertices,
which contains a $5$-cycle $C_k^{*}$ by Lemma \ref{thm:Moon}.
Note that $T[V(C_i)\cup\{x_t\}]$ is a strong subtournament of $T_n$ with $4$ vertices, which contains a $4$-cycle $C_h^{*}$.
Now $\mathcal{C}^*=\{C_1,\ldots,C_{i-1},C_{i}^*,C_{i+1},\ldots,C_{k-1},C_k^{*}\}
$ is a good collection, a contradiction.

If $p<q$,
then suppose that $x_p\in N^+_{T_n}(V(C_k))$ and $x_q\notin N^+_{T_n}(V(C_k))$.
It follows that $T[V(C_k)\cup \{x_p,x_q\}]$ is a strong subtournament of $T_n$ with at least $6$ vertices,
which contains a $5$-cycle $C_k^*$.
Note that $T[V(C_h)\cup\{x_t\}]$ is a strong subtournament of $T_n$ with at least $4$ vertices,
which contains a $4$-cycle $C_h^*$.
Now $\mathcal{C}^*=\{C_1,\ldots,C_{h-1},C_{h}^*,C_{h+1},\ldots,C_{k-1},C_k^{*}\}$ is a good collection, a contradiction.
\end{proof}

\begin{claim}\label{cl:cl6}
For every $1\leqslant i\leqslant h$ and every $h+1\leqslant j\leqslant k-1$,
we have $\lvert A(V(C_i),V(C_{j}))\rvert\leqslant 3$ for the following two cases: $(1)$ $\lvert A(x_t,V(C_i))\rvert=3$ and $\lvert A(V(C_j),X)\rvert\geqslant 4$,
$(2)$ $\lvert A(x_t,V(C_i))\rvert<3$ and $\lvert A(V(C_j),X)\rvert\geqslant 6$.
\end{claim}

\begin{proof}
It suffices to show the conclusion
for $\lvert A(x_t,V(C_i))\rvert=3$ and $\lvert A(V(C_j),X)\rvert\geqslant 4$.
If $\lvert A(x_t,V(C_i))\rvert<3$ and $\lvert A(V(C_j),X)\rvert\geqslant 6$,
then we can reverse direction of each edge and get the same conclusion.

Suppose that $\lvert A(V(C_i),V(C_{j}))\rvert\geqslant 4$,
without loss of generality,
it can be assumed that there exists $0\leqslant\ell_1\leqslant 2$
such that $d^{+}_{T[X]}(y_{\ell_1}^j)>1$.
Then there exists $p\neq t$ such that $y_{\ell_1}^j\rightarrow x_p$.
We can conclude that $V(C_j)\rightarrow \{x_{p+1},\ldots,x_t\}$.
We will get a contradiction by replacing 
$C_i$, $C_{k-1}$ with $C_i^*$, $C_{k-1}^*$
in the following.

We state that there exists $0\leqslant q\leqslant 2$ such that $y_q^i\rightarrow y_{\ell_1}^j$.
If not, then $y_{\ell_1}^j\rightarrow V(C_i)$.
Since $\lvert A(V(C_i),V(C_{j}))\rvert\geqslant 4$,
then there exists $0\leqslant\ell_2\leqslant 2$
such that $y_{\ell_2}^i\rightarrow \{y_{\ell_1+1}^j,y_{\ell_1+2}^j\}$.
If $y_{\ell_2+1}^i\rightarrow y_{\ell_1+1}^j$,
then let $C_i^*=(y_{\ell_2+1}^i,y_{\ell_1+1}^j,x_t,y_{\ell_2+1}^i)$
and $C_{j}^*=(y_{\ell_2+2}^i,y_{\ell_2}^i,y_{\ell_1+2}^j,y_{\ell_1}^j,x_p,y_{\ell_2+2}^i)$.
So $y_{\ell_1+1}^j\rightarrow y_{\ell_2+1}^i$.
If $y_{\ell_2+2}^i\rightarrow y_{\ell_1+1}^j$,
then we let
$C_i^*=(y_{\ell_2+1}^i,y_{\ell_2+2}^i,y_{\ell_1+1}^j,y_{\ell_2+1}^i)$ and
$C_{j}^*=(y_{\ell_2}^i,y_{\ell_1+2}^j,y_{\ell_1}^j,x_p,x_t,y_{\ell_2}^i)$.
It follows that $y_{\ell_1+1}^j\rightarrow \{y_{\ell_2+1}^i,y_{\ell_2+2}^i\}$.
We can conclude that $V(C_i)\rightarrow y_{\ell_1+2}^j$ by $\lvert A(V(C_i),V(C_{j}))\rvert\geqslant 4$.
Now let $C_i^*=(y_{\ell_2+1}^i,y_{\ell_2+2}^i,y_{\ell_1+2}^j,y_{\ell_1}^j,x_p,y_{\ell_2+1}^i)$ and $C_{j}^*=(y_{\ell_2}^i,y_{\ell_1+1}^j,x_t,y_{\ell_2}^i)$.

If $y_{q+1}^i\rightarrow y_{\ell_1+1}^j$, 
then $y_{q+1}^i\rightarrow y_{\ell_1+2}^j$;
otherwise, 
let $C_i^*=(y_{q+1}^i,y_{\ell_1+1}^j,y_{\ell_1+2}^j,y_{q+1}^i)$ and
$C_{j}^*=(y_{q+2}^i,y_{q}^i,y_{\ell_1}^j,x_p,x_t,y_{q+2}^i)$.
We can also conclude that $y_{\ell_1+1}^j\rightarrow y_{q+2}^i$, $y_{\ell_1}^j\rightarrow y_{q+2}^i$, $y_{\ell_1+2}^j\rightarrow y_{q}^i$, $y_{\ell_1+1}^j\rightarrow y_{q}^i$, $y_{\ell_1+1}^j\rightarrow x_p$, $x_p\rightarrow y_{\ell_1+2}^j$ 
and $y_{q+2}^i\rightarrow y_{\ell_1+2}^j$;
otherwise,
let
\begin{align*}
\begin{split}
\left\{
\begin{array}{ll}
C_i^*=(y_{q+2}^i,y_{\ell_1+1}^j,x_t,y_{q+2}^i), C_{j}^*=(y_{q}^i,y_{q+1}^i,y_{\ell_1+2}^j,y_{\ell_1}^j,x_p,y_{q}^i), &y_{q+2}^i\rightarrow y_{\ell_1+1}^j;\\
C_i^*=(y_{q+2}^i,y_{\ell_1}^j,x_p,y_{q+2}^i), C_{j}^*=(y_{q}^i,y_{q+1}^i,y_{\ell_1+1}^j,y_{\ell_1+2}^j,x_t,y_{q}^i), &y_{q+2}^i\rightarrow y_{\ell_1}^j;\\
C_i^*=(y_{q+1}^i,y_{\ell_1+1}^j,x_t,y_{q+1}^i), C_{j}^*=(y_{q+2}^i,y_{q}^i,y_{\ell_1+2}^j,y_{\ell_1}^j,x_p,y_{q+2}^i), &y_{q}^i\rightarrow y_{\ell_1+2}^j;\\
C_i^*=(y_{q+2}^i,y_{q}^i,y_{\ell_1+1}^j,y_{q+2}^i), C_{j-1}^*=(y_{q+1}^i,y_{\ell_1+2}^j,y_{\ell_1}^j,x_p,x_t,y_{q+1}^i), &y_{q}^i\rightarrow y_{\ell_1+1}^j;\\
C_i^*=(y_{q+1}^i,y_{\ell_1+2}^j,x_t,y_{q+1}^i), C_{j}^*=(y_{q+2}^i,y_{q}^i,y_{\ell_1}^j,x_p,y_{\ell_1+1}^j,y_{q+2}^i), &x_p\rightarrow y_{\ell_1+1}^j;\\
C_i^*=(y_{q+2}^i,y_{q}^i,y_{\ell_1}^j,y_{q+2}^i), C_{j}^*=(y_{q+1}^i,y_{\ell_1+1}^j,y_{\ell_1+2}^j,x_p,x_t,y_{q+1}^i), &y_{\ell_1+2}^j\rightarrow x_p;\\
C_i^*=(y_{q+1}^i,y_{\ell_1+1}^j,x_t,y_{q+1}^i),
C_{j}^*=(y_{q+2}^i,y_{q}^i,y_{\ell_1}^j,x_p,y_{\ell_1+2}^j,y_{q+2}^i),
&y_{\ell_1+2}^j\rightarrow y_{q+2}^i.
\end{array}
\right.
\end{split}
\end{align*}

If $y_{q+1}^i\rightarrow y_{\ell_1}^j$, then let $C_i^*=(y_{q}^i,y_{q+1}^i,y_{\ell_1}^j,y_{\ell_1+1}^j,x_p,y_{q}^i)$,
$C_{j}^*=(y_{q+2}^i,y_{\ell_1+2}^j,x_t,y_{q+2}^i)$.
If $y_{\ell_1}^j\rightarrow y_{q+1}^i$, then let $C_i^*=(y_{q}^i,y_{\ell_1}^j,y_{q+1}^i,y_{\ell_1+1}^j,x_p,y_{q}^i)$,
$C_{j}^*=(y_{q+2}^i,y_{\ell_1+2}^j,x_t,y_{q+2}^i)$.
It follows that $y_{\ell_1+1}^j\rightarrow y_{q+1}^i$.

If $y_{q+1}^i\rightarrow y_{\ell_1+2}^j$,
then we can get $y_{q+2}^i\rightarrow y_{\ell_1}^j$, $x_{p}\rightarrow y_{\ell_1+1}^j$, $y_{\ell_1+1}^j\rightarrow y_{q+2}^i$, $x_{p}\rightarrow y_{\ell_1+2}^j$, $y_{\ell_1+1}^j\rightarrow y_{q}^i$, $y_{\ell_1+2}^j\rightarrow y_{q}^i$.
Let $C_i^*=(y_{q}^i,y_{q+1}^i,y_{\ell_1+2}^j,y_{q}^i)$,
$C_{j}^*=(y_{q+2}^i,y_{\ell_1}^j,x_p,y_{\ell_1+1}^j,x_t,y_{q+2}^i)$.
It follows that $y_{\ell_1+2}^j\rightarrow y_{q+1}^i$.

If $y_{q+1}^i\rightarrow y_{\ell_1}^j$,
then we can get $y_{\ell_1+1}^j\rightarrow y_{q+2}^i$, $y_{\ell_1+1}^j\rightarrow y_{q}^i$, $y_{\ell_1+2}^j\rightarrow y_{q+2}^i$, $y_{\ell_1+2}^j\rightarrow y_{q}^i$.
This implies that $\lvert A(V(C_i),V(C_{j}))\rvert\leqslant 3$, a contradiction.
It follows that $y_{\ell_1}^j\rightarrow y_{q+1}^i$.

That is $V(C_j)\rightarrow y_{q+1}^i$.
Now we know that $y_{\ell_1+1}^j\rightarrow y_{q+2}^i$.
If $y_{q+2}^i\rightarrow y_{\ell_1+2}^j$,
then we let $C_i^*=(y_{q+1}^i,y_{q+2}^i,y_{\ell_1+2}^j,y_{q+1}^i)$,
$C_{j}^*=(y_{q}^i,y_{\ell_1}^j,x_p,y_{\ell_1+1}^j,x_t,y_{q+2}^i)$ or $(y_{q}^i,y_{\ell_1}^j,y_{\ell_1+1}^j,x_p,x_t,y_{q+2}^i)$.
Since $\lvert A(V(C_i),V(C_{j}))\rvert\geqslant 4$, then $y_{\ell_1+2}^j\rightarrow y_{q+2}^i$ and $y_{q+2}^i\rightarrow y_{\ell_1}^j$.
Let $C_i^*=(y_{q+1}^i,y_{q+2}^i,y_{\ell_1}^j,y_{\ell_1+1}^j,x_p,y_{q+1}^i)$ or $C_i^*=(y_{q+1}^i,y_{q+2}^i,y_{\ell_1}^j,x_p,y_{\ell_1+1}^j,y_{q+1}^i)$,
$C_{j}^*=(y_{q}^i,y_{\ell_1+2}^j,x_t,y_{q}^i)$, a good collection appears.
\end{proof}

\begin{claim}\label{cl:cl4}
$Y_2\neq \emptyset$.
\end{claim}

\begin{proof}
If $Y_2=\emptyset$,
then $A(x_t,V(C_i))\neq \emptyset$ for every $1\leqslant i\leqslant k-1$.
Now $Y_1=V(C_1)\cup\ldots\cup V(C_{k-1})$.
By Claim \ref{cl:cl1},
we have $\{x_1,\ldots,x_{t-1}\}\rightarrow Y_1$.
Then
\[
\lvert A(Y_1,V(C_k))\rvert \geqslant(2k-5)\times (2k-1)+(k+2)\times(2k-2)-\binom{3k-3}{2}=\frac{(3k-5)(k+2)}{2}.
\]
Since there are $12(k-1)$ arcs between $Y_1$ and $V(C_k)$,
we have
\[
\sum_{j=1}^{4}d_{T[Y_1]}^+(z_j)\leqslant 12(k-1)-\frac{(3k-5)(k+2)}{2}=\frac{(3k-2)(7-k)}{2}.
\]
By $\sum_{j=1}^{4}d_{T[Y_1]}^+(z_i)\geqslant 0$ and $k\geqslant 5$,
we have $5\leqslant k\leqslant 7$.

\begin{case4}\label{ca:case4.1}
There exists $1\leqslant j\leqslant 4$ such that $1 \leqslant\lvert A(x_t, V(C_j))\rvert\leqslant 2$.
\end{case4}
Without loss of generality,
suppose that $1 \leqslant\lvert A(x_t, V(C_{k-1}))\rvert\leqslant 2$.
By Claim \ref{cl:cl1}, 
we can conclude that $\{x_1,x_2,\ldots,x_{t-1}\}\rightarrow V(C_i)$
for every $1\leqslant i\leqslant k-1$.
By Claim \ref{cl:cl3}, 
we can conclude that 
$N^{+}_{T[X\backslash \{x_t\}]}(z_{j_1})=N^{+}_{T[X\backslash \{x_t\}]}(z_{j_2})$
for every $j_1,j_2\in \{1,2,3,4\}$, $j_1\neq j_2$.
It follows that $V(C_k)\rightarrow \{x_1,x_2,\ldots,x_{t-1}\}$ by the definition of $r$-outdegree-critical tournaments.
There exists $2\leqslant q\leqslant 3$ and $1\leqslant p\leqslant k-2$, such that $V(C_p)\cap N_{T_n}^-(z_q)\neq \emptyset$.
Then $T[V(C_p)\cup\{z_q\}\cup\{x_1\}]$ is a strong subtournament of $T_n$ with $5$ vertices,
which contains a $5$-cycle $C_p^*$.
Note that $T[V(C_{k-1})\cup\{x_t\}]$ is a strong subtournament of $T_n$ with $4$ vertices,
which contains a $4$-cycle $C_{k-1}^*$.
It is not difficult to see that $C_k^*=(z_1,z_{3-q},z_4,z_1)$ is a $3$-cycle,
then \[\mathcal{C}^*=\{C_1,C_2,\ldots,C_{p-1},C_p^*,C_{p+1},\ldots,C_{k-1}^*,C_k^*\}\] is a good collection, a contradiction.

\begin{case4}\label{sca:subcase4.1.2}
$\lvert A(x_t,V(C_j))\rvert=3$ for every $1\leqslant j\leqslant k-1$.
\end{case4}

This implies that $x_t\rightarrow Y_1$ and
\[
\sum_{i=1}^{4}d^{+}_{T[Y_1]}(z_i)
\leqslant4\times (3k-3)-\left((3k-3)\times (2k-1)-\binom{3k-3}{2}\right)\leqslant 6.
\]
It is not difficult to find a good collection,
a contradiction.
\end{proof}

\begin{claim}\label{cl:cl5}
The tournament $T_n$ is not $(2k-1)$-regular.
\end{claim}

\begin{proof}
If $T_n$ is $(2k-1)$-regular,
then $n=4k-1$, $t=k-2$
and $d_{T_n}^+(v)=d_{T_n}^-(v)=2k-1$ for every vertex $v\in V(T_n)$.
It follows that
$d_{T_n}^+(x_1)=d_{T[X]}^+(x_1)+d_{T[V(\mathcal{C})]}^+(x_1)\geqslant k-3+3h$
and thus $3h\leqslant k+2$.
Since $d_{T_n}^+(x_t)=2k-1$,
we have $h\geqslant \frac{2k-5}{3}$.
Thus $5\leqslant k\leqslant 7$.
Now we distinguish three cases to prove this claim.

\begin{case5}\label{ca:case5.1}
$k=5$.
\end{case5}

Then $n=19$, $t=3$
and $d_{T_{19}}^+(v)=d_{T_{19}}^-(v)=9$ for every vertex $v\in V(T_{19})$.
This implies that
$x_1\rightarrow V(C_1)\cup V(C_2)\cup\{x_2,x_3\}$ and $x_2\rightarrow V(C_1)\cup V(C_2)\cup\{x_3\}$.
Then $d_{T[V(C_3)\cup V(C_4)\cup V(C_5)]}^+(x_1)=1$
and $d_{T[V(C_3)\cup V(C_4)\cup V(C_5)]}^+(x_2)=2$.
It follows that $\{V(C_3), V(C_4)\}\rightarrow x_3$.

If $\lvert A(x_3,V(C_1))\rvert=2$, then it follows that $A(x_3,V(C_3)\cup V(C_4))\neq\emptyset$,
a contradiction.
Then $\lvert A(x_3,V(C_1))\rvert=3$.
It follows that $T[V(C_5)\cup\{x_3\}]$ is a strong subtournament of $T_{19}$ with $5$ vertices,
which contains a $5$-cycle $C_5^{*}$.

We show that $N^+_{T[X]}(y_i^3)=N^+_{T[X]}(y_j^3)$,
for every $i,j\in\{0,1,2\}$ and $i\neq j$.
If not,
then $T[V(C_3)\cup\{x_1,x_2\}]$ contains a $4$-cycle $C_3^{*}$.
It follows that $\mathcal{C}^*=\{C_1,C_2,C_3^*,C_4,C_5^*\}$ is a good collection, a contradiction.
Since $d_{T[V(C_3)\cup V(C_4)\cup V(C_5)]}^+(x_1)=1$
and $d_{T[V(C_3)\cup V(C_4)\cup V(C_5)]}^+(x_2)=2$,
then $V(C_3)\rightarrow \{x_1,x_2\}$.
For every $0\leqslant q\leqslant 2$,
there exists $0\leqslant p\leqslant 2$
such that $N^{+}_{T[V(C_3)]}(y_q^1)=\{y_p^3\}$.
If not,
suppose that $y_0^1\rightarrow y_p^3$ and $y_1^1\rightarrow y_{p_1}^3$,
for some $p,p_1\in\{0,1,2\}$ and $p_1\neq p$.
Let $C_1^*=(y_2^1,y_0^1,y_p^3,x_1,y_2^1)$ and $C_2^*=(y_1^1,y_{p_1}^3,x_2,y_1^1)$.
Now $\mathcal{C}^*=\{C_1^{*},C_2^*,C_3,C_4,C_5^*\}$ is a good collection, 
a contradiction.
So only $y_p^3$ has in-neighbors in $V(C_1)$,
this implies that every vertex in $V(C_3)\backslash\{y_p^3\}$ has at least $10$ out-neighbors in $T_{19}$,
a contradiction to the fact that $T_{13}$ is a $9$-regular tournament.

\begin{case5}\label{ca:case5.2}
$k=6$.
\end{case5}
Then $n=23$, $t=4$ and 
$d_{T_{23}}^+(v)=d_{T_{23}}^-(v)=11$ for every vertex $v\in V(T_{23})$.
It follows that $x_i\rightarrow V(C_1)\cup V(C_2)\cup V(C_3)\cup\{x_{i+1},\ldots,x_4\}$, where $1\leqslant i\leqslant 3$.
This implies that $d_{T_{23}}^+(x_1)\geqslant 12$,
a contradiction to the fact that $T_{23}$ is a $11$-regular tournament.

\begin{case5}\label{ca:case5.3}
$k=7$.
\end{case5}

Then $n=27$, $t=5$
and $d_{T_{27}}^+(v)=d_{T_{27}}^-(v)=13$ for every vertex $v\in V(T_{27})$.
This implies that $x_t\rightarrow V(C_1)\cup V(C_2)\cup V(C_3)\cup V(C_7)$,
otherwise we can get the same contradiction as in Case \ref{ca:case5.2}.
Then $x_i\rightarrow V(C_1)\cup V(C_2)\cup V(C_3)\cup\{x_{i+1},\ldots,x_5\}$ by Claim \ref{cl:cl1}, for every $1\leqslant i\leqslant 4$.
It follows that $V(C_4)\cup V(C_5)\cup V(C_6)\rightarrow \{x_1,\ldots,x_5\}$.
For every $q\in \{0,1,2\}$, there exists $0\leqslant p\leqslant 2$, such that $N^{+}_{T[V(C_4)]}(y_q^1)=\{y_p^4\}$.
If not, suppose that $y_0^1\rightarrow y_p^4$,
then $y_1^1\rightarrow y_{p_1}^4$ for some $p,p_1\in\{0,1,2\}$ and $p_1\neq p$.
Let $C_1^*=(y_2^1,y_0^1,y_p^4,x_1,x_2,y_2^1)$ and 
$C_4^*=(y_1^1,y_{p_1}^4,x_3,y_1^1)$.
Then $\mathcal{C}^*=\{C_1^{*},C_{2},C_{3},C_4^{*},C_5,C_6,C_7\}$ is a good collection, a contradiction.
So only $y_p^4$ has in-neighbors in $V(C_1)$,
it follows that every vertex in $V(C_4)\backslash\{y_p^4\}$ has $3$ out-neighbors in $T[V(C_1)]$.
Similarly, we can conclude that $\lvert A(V(C_i),V(C_j))\rvert\leqslant 3$ for every $1\leqslant i\leqslant 3$, $4\leqslant j\leqslant 6$.
Then
\[
\sum_{i=1}^{4}d^{+}_{T[V(C_1)\cup V(C_2)\cup V(C_3)]}(z_i)
\leqslant4\times 9-\left(9\times 13-\binom{9}{2}-3\times 3\right)
<0,
\]
a contradiction.
\end{proof}

\begin{claim}\label{cl:cl7}
$x_t\rightarrow Y_1$.
\end{claim}

\begin{proof}
If not,
then there exists $1\leqslant i\leqslant h$
such that $\lvert A(V(C_i),x_t)\rvert\neq\emptyset$.
We can conclude that
$V(C_k)\rightarrow \{x_1,x_2,\ldots,x_{t-1}\}$ by Claim \ref{cl:cl3}.

We state that
$\lvert A(V(C_j),V(C_{k}))\rvert\leqslant 4$ for every $1\leqslant j\leqslant h$;
otherwise, there exists $1\leqslant p\leqslant h$
such that $\lvert A(V(C_p),V(C_{k}))\rvert\geqslant 5$.
It follows that $T[V(C_p)\cup V(C_k)\cup\{x_1,x_2,x_3\}]$
contains a $4$-cycle and a $5$-cycle, a contradiction.

If there exists $h+1\leqslant q\leqslant k-1$
such that $\lvert A(V(C_q),X)\rvert\geqslant 6$,
then $\lvert A(V(C_i),V(C_q))\rvert\leqslant 3$
for every $1\leqslant i\leqslant h$.
Therefore,
\begin{align*}
\lvert A(T[V(\mathcal{C})])\rvert&=\binom{3k+1}{2}=\lvert A(T[Y_1])\rvert+\lvert A(T[Y_2\cup Z])\rvert+\lvert A(Y_1,Y_2\cup Z)\rvert+\lvert A(Y_2\cup Z,Y_1)\rvert\\&\geqslant 3h\times(2k-1)-(3h-2k+5)+\binom{3k+1-3h}{2}+6h+8h.
\end{align*}
A simple calculation yields that $9h^2-6kh-13h+4k^2-10\leqslant 0$,
this requires $36k^2-372k+649\geqslant 0$.
So $k\leqslant 2$,
a contradiction to $k\geqslant5$.
\end{proof}

\begin{claim}\label{cl:cl8}
At most two cycles of length $3$ have out-neighbors in $\{x_1,x_2,\ldots,x_{t-1}\}$.
At least one cycle of length $3$ has out-neighbors in $\{x_1,x_2,\ldots,x_{t-1}\}$.
\end{claim}

\begin{proof}
Suppose that there are three cycles $C_{r_1}$, $C_{r_2}$ and $C_{r_3}$,
which all have out-neighbors in $\{x_1,x_2,\ldots,x_{t-1}\}$,
here $r_1,r_2\in\{h+1,\ldots,k-1\}$ and $r_1,r_2,r_3$ are three distinct numbers.
By Claim \ref{cl:cl6},
we have $\lvert A(V(C_i),V(C_{r_j}))\rvert\leqslant 3$
for every $1\leqslant i\leqslant h$ and $j=1,2,3$.
Then
\begin{align*}
\lvert A(T[V(\mathcal{C})])\rvert&=\binom{3k+1}{2}=\lvert T[Y_1]\rvert+\lvert A(T[Y_2\cup Z])\rvert+\lvert A(Y_2\cup Z,Y_1)\rvert+\lvert A(Y_1,Y_2\cup Z)\rvert\\
&\geqslant 3h\times(2k-1)+\binom{3k+1-3h}{2}+3\times 6h.
\end{align*}
A simple calculation yields that $h\leqslant\frac{2k-9}{3}$,
a contradiction to $h\geqslant\frac{2k-5}{3}$.

Next we suppose that no cycle of length $3$ has out-neighbors in $\{x_1,x_2,\ldots,x_{t-1}\}$.
Then we can conclude that
\[
    \sum_{y\in Y)}d^+_{T[Y]}(y)= \binom{3k-3}{2},
\]
$$\sum_{y\in Y}d^+_{T[V(\mathcal{C})]}(y)\geqslant 3h\times (2k-1)+3(k-1-h)\times (2k-2),$$
$$0\leqslant\sum_{i=1}^4 d^+_{T[Y]}(z_i)\leqslant 4\times (3k-3)-(6k^2-12k+3h+6)+\frac{9k^2-21k+12}{2}.$$
A simple calculation yields that $1\leqslant k\leqslant\frac{14}{3}$,
which contradicts $k\geqslant5$.
\end{proof}

Then we divide the discussion into the following two cases:
\begin{case0}
There are two $3$-cycles that have out-neighbors in $\{x_1,x_2,\ldots,x_{t-1}\}$.
\end{case0}
By renaming the vertices, these cycles can be set to $C_{k-2}$ and $C_{k-1}$.
We can conclude that
$$0\leqslant\sum_{i=1}^4 d^+_{T[Y]}(z_i)\leqslant 4\times (3k-3)-(6k^2-7k-4)+\frac{9k^2-21k+12}{2}.$$
$$0\leqslant\sum_{i=1}^4 d^+_{T[Y\backslash V(C_{k-1})]}(z_i)\leqslant 4\times (3k-9)-(6k^2-7k-4)+\frac{9k^2-21k+12}{2}.$$
Then by simple calculation we get $5\leqslant k\leqslant 7$.
We divide the discussion into the following three subcases.

\begin{subcase0}
Both $C_{k-2}$ and $C_{k-1}$ have out-neighbors in $\{x_1,x_2,\ldots,x_{t-2}\}$.
\end{subcase0}

It follows that
\[\sum_{y\in V(C_1)\cup \cdots \cup V(C_{k-3})} d^+_{T[\mathcal{C}]}(y)\geqslant 3h\times (2k-1)+3(k-3-h)\times(2k-2)\geqslant6k^2-22k+13,\]
and
\[
\lvert A(T[V(C_1)\cup \cdots \cup V(C_{k-3})])\rvert =\binom{3k-9}{2}=\frac{9}{2}k^2-\frac{57}{2}k+45.
\]
By Claim \ref{cl:cl6}, we can conclude that 
\[
\lvert A(V(C_1)\cup \cdots \cup V(C_{k-3}),V(C_{k-2})\cup V(C_{k-1}))\rvert\leqslant 2\times 3\times (k-3)=6k-18.
\]
Then 
\begin{align*}
\sum_{i=1}^4 d^+_{T[V(C_1)\cup \cdots \cup V(C_{k-3})]}(z_i)
&\leqslant 4\times (3k-9)-\left(6k^2-22k+13-\frac{9}{2}k^2+\frac{57}{2}k-45-6k+18\right)\\
&=-\frac{3}{2}k^2+\frac{23}{2}k-22.
\end{align*}
A simple calculation yields that $k\leqslant 4$,
a contradiction to $k\geqslant5$.

\begin{subcase0}
$C_{k-2}$ has no out-neighbors in $\{x_1,x_2,\ldots,x_{t-2}\}$, but $C_{k-1}$ has out-neighbors in $\{x_1,x_2,\ldots,x_{t-2}\}$.
\end{subcase0}

Then
\begin{align*}
\lvert A(T[V(\mathcal{C})])\rvert
&=\lvert A(T[V(C_1)\cup \cdots \cup V(C_{k-3})])\rvert+\lvert A(T[V(C_{k-2})\cup V(C_{k-1})\cup Z])\rvert\\
&+\lvert A(V(C_1)\cup \cdots \cup V(C_{k-3}),V(C_{k-2})\cup V(C_{k-1})\cup Z)\rvert\\
&+\lvert A(V(C_{k-2})\cup V(C_{k-1})\cup Z,V(C_1)\cup \cdots \cup V(C_{k-3}))\rvert\\
&\geqslant 3h\times(2k-1)+3(k-3-h)\times(2k-2)+\binom{10}{2}+6\times (k-3)+6\times h.
\end{align*}
A contradiction.

\begin{subcase0}
Neither $C_{k-2}$ nor $C_{k-1}$ has out-neighbors in $\{x_1,x_2,\ldots,x_{t-2}\}$.
\end{subcase0}
It follows that
\[
\sum_{y\in Y} d^+_{T[\mathcal{C}]}(y)\geqslant 3h\times (2k-1)+3(k-3-h)\times(2k-2)+6\times (2k-3)\geqslant6k^2-10k-5,
\]
and
\[
\lvert A(T[Y])\rvert =\binom{3k-3}{2}=\frac{9}{2}k^2-\frac{21}{2}k+6.
\]
Then 
\begin{align*}
\sum_{i=1}^4 d^+_{T[Y]}(z_i)
&\leqslant 4\times (3k-3)-\left(6k^2-10k-5-\frac{9}{2}k^2+\frac{21}{2}k-6\right)\\
&=-\frac{3}{2}k^2+\frac{23}{2}k-1.
\end{align*}
A simple calculation yields that $5\leqslant k\leqslant 7$,
then we divide the discussion into the following three subcases:

$(i)$~$k=5$.

It follows that $h=2$,
then 
\[
\sum_{i=1}^4 d^+_{T[Y_1]}(z_i)\leqslant 4\times 6-\left(6\times 9-\binom{6}{2}-3\times2\times2\right)<0,
\]
a contradiction.

$(ii)$~$k=6$.

It follows that $h=3$,
then
\[
\sum_{i=1}^4 d^+_{T[Y_1]}(z_i)\leqslant 4\times 9-\left(9\times 11-\binom{9}{2}-3\times2\times3\right)<0,
\]
a contradiction.

$(iii)$~$k=7$.

Then
\[
\sum_{i=1}^4 d^+_{T[Y]}(z_i)\leqslant 4\times 18-\left(9\times 13+3\times12+6\times11-\binom{18}{2}\right)\leqslant 6.
\]
It follows that $\lvert A(V(C_i),Z)\rvert\geqslant 6$, for every $1\leqslant i\leqslant 6$.
This implies that $z_i$ has at least $5$ out-neighbors in $X$,
for every $1\leqslant i\leqslant 4$.
It is not difficult to find a good collection,
a contradiction.

\begin{case0}
There is a $3$-cycle that has out-neighbors in $\{x_1,x_2,\ldots,x_{t-1}\}$.
\end{case0}

By renaming the vertices, this cycle can be set to $C_{k-1}$.
Then we divide the discussion into the following two subcases.

\begin{subcase0}
$C_{k-1}$ has out-neighbors in $\{x_1,x_2,\ldots,x_{t-2}\}$.
\end{subcase0}

Then $\lvert A(V(C_i),V(C_{k-1}))\rvert\leqslant3$ by Claim \ref{cl:cl6}.
We can conclude that
\[
0\leqslant\sum_{i=1}^4 d^+_{T[V(C_1)\cup\cdots\cup V(C_{k-2})]}(z_i)\leqslant 15\times (k-2)-3h\times (2k-1)-3(k-2-h)\times (2k-2)+\binom{3k-6}{2}.
\]
A simple calculation yields that $3k^2-11k+32\leqslant 0$,
a contradiction.

\begin{subcase0}
$C_{k-1}$ has no out-neighbor in $\{x_1,x_2,\ldots,x_{t-2}\}$.
\end{subcase0}

We can conclude that
\begin{align*}
\lvert A(T[V(\mathcal{C})])\rvert
&=\binom{3k+1}{2}=\lvert T[Y]\rvert+\lvert T[Z]\rvert+\lvert A(Y,Z)\rvert+\lvert A(Z,Y)\rvert\\
&\geq 3h\times(2k-1)+3(k-2-h)\times (2k-2)+3\times(2k-3)+\binom{4}{2}.
\end{align*}
This implies that $5\leqslant k\leqslant 7$.
By the definition of $r$-outdegree-critical tournaments,
we can conclude that there exists a vertex $w\in V(C_k)$ with $w\rightarrow x_1$.
Then we divide the discussion into the following three cases.

$(i)~k=5$.

If $h=3$, then
\[
\sum_{i=1}^4 d^+_{T[Y_1]}(z_i)\leqslant 4\times 9-\left(9\times 9-3\times 3-\binom{9}{2}\right)=0.
\]
It follows that $Y_1\rightarrow Z$.
This implies that $z_i$ has at least $4$ out-neighbors in $X$,
for every $1\leqslant i\leqslant 4$.
Then we assume that
\begin{align*}
\begin{split}
\left\{
\begin{array}{ll}
x_{p_1}\in N_{T[X]}^+(z_1), &p_1\neq t;\\
x_{p_2}\in N_{T[X]}^+(z_2), &p_2\neq p_1,t;\\
x_{p_3}\in N_{T[X]}^+(z_3), &p_3\neq p_1,p_2,t.
\end{array}
\right.
\end{split}
\end{align*}
Then $C_2^*=(y_0^2,z_1,x_{p_1},y_0^2)$ is a $3$-cycle,
and $C_5^*=(y_1^2,y_2^2,z_2,x_{p_2},y_1^2)$ is a $4$-cycle.
Note that $T[V(C_3)\cup \{z_3\}\cup \{x_{p_3}\}]$ is a strong subtournament of $T_n$ with $5$ vertices,
which contains a $5$-cycle $C_3^*$.
Now $\mathcal{C}^*=\{C_1,C_2^*,C_3^*,C_4,C_5^*\}$ is a good collection,
a contradiction.

If $h=2$, 
then
\[
\sum_{i=1}^4 d^+_{T[Y_1\cup V(C_3)]}(z_i)\leqslant 4\times 9-\left(6\times 9+3\times 8-3\times 3-2\times 3-\binom{9}{2}\right)=9.
\]
It follows that $\lvert A(V(C_i),Z)\rvert\geqslant 3$ for every $1\leqslant i\leqslant 3$.
We can get a contradiction similarly.

$(ii)~k=6$.

It follows that
\[
\sum_{i=1}^4 d^+_{T[Y]}(z_i)\leqslant 4\times 15-\left(9\times 11+3\times 10+3\times 9-\binom{15}{2}\right)=9
\]
and $\lvert A(V(C_i),Z)\rvert\geqslant 3$ 
for every $1\leqslant i\leqslant 5$.
Similar to the case $(i)$, we can get a contradiction.

$(iii)~k=7$.

It follows that
\[
\sum_{i=1}^4 d^+_{T[Y]}(z_i)\leqslant 4\times 18-\left(9\times 13+6\times 12+3\times 11-\binom{18}{2}\right)=3.
\]
Then $\lvert A(V(C_i),Z)\rvert\geqslant 9$
for every $1\leqslant i\leqslant 6$.
Similar to the case $(i)$, we can get a contradiction.

The proof of Theorem \ref{thm:main} is complete.

\section{Proof of Theorem \ref{thm:main2}}
\label{section:proof1}

We will also present the proof by contradiction.
Suppose that $T_n$ is a $6$-outdegree-critical tournament on $n$ vertices
but contains no good collection of vertex-disjoint cycles.
Recall that $h^*(k,2)\leqslant 2k-1$.
We can get that $T_n$ contains $3$ vertex-disjoint cycles $C_1$, $C_2$, $C_3$, in which two of them have different lengths.
Without loss of generality, assume that $\lvert V(C_1)\rvert=\lvert V(C_2)\rvert=3$ and $\lvert V(C_3)\rvert=4$.
Let $C_i=(y_0^i,y_1^i,y_2^i,y_0^i)$, where $1\leqslant i\leqslant 2$, $C_{3}=(z_1,z_2,z_3,z_4,z_1)$.
Up to isomorphism,
assume that $z_1\rightarrow z_3$ and $z_2\rightarrow z_4$.
Let $V(\mathcal{C})=V(C_1)\cup V(C_2)\cup V(C_{3})$, $X=V(T_n)\backslash V(\mathcal{C})$, $Y=V(\mathcal{C})\backslash V(C_3)$ and $Z=V(C_3)$.
It is not difficult to see that
$\lvert V(\mathcal{C})\rvert=10$
and $\lvert X\rvert=n-10$,
denote $\lvert X\rvert$ by $t$.

We can verify that Claims 1-5 and Claim 7 are also true.
The desired contradiction will appear after the following two claims.

\begin{claim}\label{cl:cl9}
The tournament $T_n$ is not $6$-regular.
\end{claim}

\begin{proof}
If $T_n$ is $6$-regular,
then $n=13$, $t=3$
and $d_{T_{13}}^+(v)=d_{T_{13}}^-(v)=6$ for every vertex $v\in V(T_{13})$.
This implies that
$x_1\rightarrow V(C_1)\cup\{x_2,x_3\}$ and $x_2\rightarrow V(C_1)\cup\{x_3\}$.
It follows that $d_{T[V(C_2)\cup V(C_3)]}^+(x_1)=1$
and $d_{T[V(C_2)\cup V(C_3)]}^+(x_2)=2$.

\begin{case10}\label{sca:subcase10.1.1}
$\lvert A(x_3,V(C_1))\rvert=2$.
\end{case10}

It follows that $x_3\rightarrow V(C_3)$
and $N^{+}_{T[X]}(z_i)=N^{+}_{T[X]}(z_j)$,
where $i,j\in \{1,2,3,4\}$ and $i\neq j$,
that is $V(C_3)\rightarrow \{x_1,x_2\}$.
Then $T[V(C_2)\cup \{x_1,x_{2}\}]$ is a strong subtournament of $T_{13}$ with $5$ vertices,
which contains a $5$-cycle $C_2^*$.
Now $\mathcal{C}^*=\{C_1,C_2^*,C_3\}$ is a good collection, a contradiction.

\begin{case10}\label{sca:subcase10.1.2}
$\lvert A(x_3,V(C_1))\rvert=3$.
\end{case10}

It follows that $T[V(C_3)\cup\{x_3\}]$ is a strong subtournament of $T_{13}$ with $5$ vertices,
which contains a $5$-cycle $C_3^{*}$.
We show that $N^+_{T[X]}(y_i^2)=N^+_{T[X]}(y_j^2)$,
here $i,j\in\{0,1,2\}$ and $i\neq j$.
If not,
then $T[V(C_2)\cup\{x_1,x_2\}]$ contains a $4$-cycle $C_2^{*}$.
Now $\mathcal{C}^*=\{C_1,C_2^*,C_3^*\}$ is a good collection, a contradiction.
Since $d^{+}_{T[V(C_2)\cup V(C_3)]}(x_1)=1$
and $d^{+}_{T[V(C_2)\cup V(C_3)]}(x_2)=2$,
then $V(C_2)\rightarrow \{x_1,x_2\}$.
For every $0\leqslant q\leqslant 2$,
there exists $0\leqslant p\leqslant 2$
such that $N^{+}_{T[V(C_2)]}(y_q^1)=\{y_p^2\}$.
If not,
suppose that $y_0^1\rightarrow y_p^2$ and $y_1^1\rightarrow y_q^2$,
where $p,q\in\{0,1,2\}$ and $q\neq p$.
Let $C_1^*=(y_2^1,y_0^1,y_p^2,x_1,y_2^1)$ and $C_2^*=(y_1^1,y_q^2,x_2,y_1^1)$.
Then $\mathcal{C}^*=\{C_1^{*},C_2^*,C_3^*\}$ is a good collection, a contradiction.
So only $y_p^2$ has in-neighbors in $V(C_1)$,
this implies that every vertex in $V(C_2)\backslash\{y_p^2\}$ has at least $7$ out-neighbors in $T_{13}$,
a contradiction to the fact that $T_{13}$ is a $6$-regular tournament.
\end{proof}

\begin{claim}\label{cl:cl10}
No cycle of length $3$ has out-neighbors in $\{x_1,x_2,\ldots,x_{t-2}\}$.
\end{claim}

\begin{proof}
If $C_2$ has out-neighbors in $\{x_1,x_2,\ldots,x_{t-2}\}$,
then by Claim \ref{cl:cl6},
we have $\lvert A(V(C_1),V(C_{2}))\rvert\leqslant 3$.
It follows that
\[
    \sum_{i=1}^4 d^+_{T[V(C_{1})]}(z_i)\leqslant 4\times 3-\left(3\times 6-3-3\right)=0.
\]
It is not difficult to find a good collection, a contradiction.
\end{proof}

By the definition of $r$-outdegree-critical tournaments,
there exists $z'\in V(C_3)$ by Claim \ref{cl:cl8},
such that $d_{T_n}^+(z')=6$ and $z'\rightarrow x_1$.
Since $d_{T_n}^+(y_i^1)\geqslant 6$,
we have $d_{T[V(C_2)]}^+(y_i^1)\geqslant 1$ for every $0\leqslant i\leqslant 2$.
Now we state that $\{x_1,x_2,\ldots,x_{t-1}\}\rightarrow V(C_2)$.
If not, this implies that $\lvert A(V(C_1),V(C_{2}))\rvert\leqslant 3$ by Claim \ref{cl:cl6},
then $d^{+}_{T[V(C_2)]}(y_i^1)=1$ and $V(C_1)\rightarrow V(C_3)$.
So $T[V(C_2)\cup \{y_0^1,x_{t}\}]$ is a strong subtournament of $T_n$ with $5$ vertices,
which contains a $5$-cycle $C_{2}^{*}$.
Note that $d^{+}_{T[Y]}(z_j)\leqslant 3$ for every $1\leqslant j\leqslant 4$,
then $V(C_3)\backslash \{z'\}$ have different out-neighbors $x_p$ than $x_1$.
So $T[V(C_3)\backslash \{z'\}\cup \{y_1^1,x_{p}\}]$ contains a $4$-cycle $C_{3}^{*}$.
Let $C^{*}=(y_2^1,z',x_1,y_2^1)$.
Then $\mathcal{C}^*=\{C_1^{*},C_2^*,C_3^*\}$ is a good collection, a contradiction.
There exists $0\leqslant p\leqslant 2$,
such that $d^{+}_{T[V(C_2)]}(y_p^1)\geqslant 2$,
a contradiction to $\lvert A(V(C_1),V(C_{2}))\rvert\leqslant 3$.
Then $\{x_1,x_2,\ldots,x_{t-1}\}\rightarrow V(C_2)$.
Therefore,
$$\sum_{j=1}^{4}d^{+}_{T[V(\mathcal{C})]}(z_j)\leqslant 45-3\times 6-3\times5=12,$$
such that $$d^{+}_{T[Y]}(z_2)+d^{+}_{T[Y]}(z_3)\leqslant 6.$$
It follows that $d^{-}_{T[Y]}(z_q)\geqslant 1$ and $d^+_{T[X]}(z_q)\geqslant 2$, where $2\leqslant q\leqslant 3$.
Suppose that $z_q$ has a different out-neighbor $x_p$ in $X$ other than $x_t$.

Suppose that there exists $u\in V(C_r)$ with $u\rightarrow z_q$, where $1\leqslant r\leqslant 2$.

Firstly, we prove that $z_q\rightarrow \{x_{t-1},x_t\}$.
If $x_t\rightarrow z_q$, then $T[V(C_1)\cup V(C_2)\backslash\{u\}\cup\{x_{t}\}]$ is a strong subtournament of $T_n$ with at least $6$ vertices,
which contains a $5$-cycle $C^{*}$.
Note that $T[\{u,z_q\}\cup N^+_{T[X]}(z_q)]$ is a strong subtournament of $T_n$ with $4$ vertices,
which contains a $4$-cycle $C^{**}$.
Let $C_3^{*}=(z_1,z_{5-i},z_4,z_1)$.
Then $\mathcal{C}^*=\{C^{*},C^{**},C_3^*\}$ is a good collection, a contradiction.
It follows that $z_i\rightarrow x_t$.
If $x_{t-1}\rightarrow z_q$, 
then assume that $z_q\rightarrow x_p$, where $p<t-1$.
Replace $C^{**}$ with $(u,z_q,x_p,x_{p+1},u)$, 
we can deduce a contradiction, so $z_q\rightarrow x_{t-1}$.

Then we can prove that $\{z_{q+1},z_{q+2},z_{q+3}\}\rightarrow \{x_2,\ldots,x_{t-1}\}$.
Note that $C^{'}=(z_{q+1},z_{q+2},\\z_{q+3},z_{q+1})$ is a $3$-cycle
and $x_1\in \{N^+_{T_n}(z_{q+1})\cup N^+_{T_n}(z_{q+2})\cup N^+_{T_n}(z_{q+3})\}$.
If there exists $2\leqslant \ell\leqslant t-2$ such that $d^+_{\{z_{q+1},z_{q+2},z_{q+3}\}}(x_{\ell})\geqslant 1$,
then $T[\{x_1,\ldots,x_{\ell}\}\cup \{z_{q+1},z_{q+2},z_{q+3}\}]$ is a strong subtournament of $T_n$ with at least $5$ vertices,
which contains a $5$-cycle $C^{*}$.
Note that $T[V(C_r)\cup\{z_q\}\cup \{x_{t-1}\}]$ is a strong subtournament of $T_n$ with at least $5$ vertices,
which contains a $4$-cycle $C^{**}$.
Now $\mathcal{C}^*=\{C^{*},C^{**},C_{3-r}\}$ is a good collection, 
a contradiction.

It is not difficult to see that $T[\{x_{t-2},x_t\}\cup \{z_{q+1},z_{q+2},z_{q+3}\}]$ is a strong subtournament of $T_n$ with $5$ vertices,
which contains a $5$-cycle $C^{*}$.
Note that $T[V(C_r)\cup\{z_q\}\cup \{x_{t-1}\}]$ is a strong subtournament of $T_n$ with at least $5$ vertices,
which contains a $4$-cycle $C^{**}$.
Now $\mathcal{C}^*=\{C^{*},C^{**},C_{3-r}\}$ is a good collection, 
a contradiction.

The proof of Theorem \ref{thm:main2} is complete.

\section{Concluding remarks}\label{section:CR}

In this paper,
we first introduce a new function $h(k,\ell)$ 
which generalizes the functions defined in  
Conjecture \ref{conj:BTC} and Conjecture \ref{conj:LC} simultaneously.
Afterwards,
we concentrate on finding vertex-disjoint cycles 
of different lengths in tournaments.
We improve a main result of Chen and Chang in \cite{CC} by 
showing that when $k\geqslant 5$ every tournament with mminimum outdegree at least $2k-1$ contains $k$ vertex-disjoint cycles in which
three of them have different lengths,
i.e., $h^*(k,3)\leqslant 2k-1$ for $k\geqslant 5$.
In addition, we show that $h^*(3,3)=6$,
which also answers a question proposed by Chen and Chang in \cite{CC}.

For every digraph containing $k$ vertex-disjoint cycles 
in which $\ell$ of them have different lengths, 
the sum of the lengths of these cycles is clearly not less than 
the sum of the numbers in $\{3,4,\ldots,\ell+2\}$. 
Also, the sum of the lengths of the remaining $k-\ell$ cycles 
is not less than $3(k-\ell)$. 
Thus the order of such a digraph is at least 
$3(k-\ell)+\sum_{i=3}^{\ell-2} i$.
Since the minimum outdegree of each tournament 
is less than half of its order,
the following lower bound of $h^*(k,\ell)$ holds, 
\[
h^*(k,\ell)
\geqslant \frac{1}{2} \left(3(k-\ell)-1+\sum_{i=3}^{\ell+2} i\right)
=\frac{3k-1}{2}+\frac{\ell^2-\ell}{4}.
\]
For $k$-regular tournaments with small values of $k$,
one can see that each contains exactly $2k+1$ vertices
and thus cannot contain too many vertex-disjoint cycles of different lengths.
So the equality of the above lower bound may not hold.
But for large values of $k$,
it would be interesting to decide 
whether the equality holds. 
We propose the following problem.

\begin{problem}\label{prob:cdl}
For any positive integer $\ell$,
does there exist an integer $k_{\ell}$ such that,
for each $k\geqslant k_{\ell}$,
every tournament with minimum outdegree at least 
$\frac{3k-1}{2}+\frac{\ell^2-\ell}{4}$
contains $k$ vertex-disjoint cycles in which $\ell$ of them have different lengths?
\end{problem}

One main result of Bang-Jensen et al. in \cite{BBT} 
implies that the answer for Problem \ref{prob:cdl} is `yes' for $\ell=1$.
Note that the minimum outdegree of the above tournament is equal to half of the sum of the number of vertices of the cycle collection.
It would be interesting to check 
if for sufficiently large $k$,
for any $k$ different positive integers $a_1, \ldots, a_k$, 
every digraph with minimum outdegree at least $\frac{1}{2}(\sum_{i=1}^k a_i-1)$ 
contains $k$ vertex-disjoint cycles, 
whose lengths are $a_1, \ldots, a_k$, respectively.
We propose the problem below.

\begin{problem}
For any $\alpha\geqslant 3$,
does there exist an integer $k_{\alpha}$ such that,
for any $k\geqslant k_{\alpha}$ 
and any $k$ integers $a_1, \ldots, a_k\in\{3,4,\ldots,\alpha\}$,
every tournament with minimum outdegree at least 
$\frac{1}{2}(\sum_{i=1}^k a_i-1)$ contains $k$ vertex-disjoint cycles 
whose lengths are $a_1, \ldots, a_k$, respectively?
\end{problem}

Note also that for each positive integer $d$
every digraph with minimum outdegree at least $d$
contains $d$ cycles of different lengths.
So in the setting of tournaments,
the finiteness of $h^*(k,\ell)$
follows directly from the splitting results in \cite{ABB, YBWW}.
But in the setting of general digraphs,
few results on the finiteness of $h(k,\ell)$ are known.
One can easily check that $h(1,1)=1$. 
Thomassen \cite{T} proved that $h(2,1)=3$. 
Lichiardopol et al. \cite{LPS}, 
and Bai and Manoussakis \cite{BM} showed that $h(3,1)=5$.
Lichiardopol \cite{L} proved that $h(2,2)=4$ and,
moreover, conjectured that $h(k,k)$ is finite for each $k$.
This conjecture is widely open for any $k\geqslant 3$.
In fact, there is not even any partial progress on this problem in the past ten years.
Here we propose a weaker conjecture for further research.

\begin{conjecture}
For every positive integer $\ell$,
if the integer $k$ is sufficiently large compared to $\ell$,
then $h(k,\ell)$ is finite.
\end{conjecture}


\begin{thebibliography}{99}

\bibitem{A} 
N. Alon, 
\emph{Disjoint directed cycles}, 
J. Combin. Theory Ser. B. 68 (1996), 167-178.

\bibitem{ABB} 
N. Alon, J. Bang-Jensen, S. Bessy, 
\emph{Out-colourings of digraphs}, 
J. Graph Theory. 93 (1) (2020), 88-112.

\bibitem{BLL} Y. Bai, B. Li, H. Li, \emph{Vertex-disjoint cycles in bipartite tournaments}, Discrete Math. 338 (2015), 1307-1309.
\bibitem{BM} Y. Bai, Y. Manoussakis, \emph{On the number of vertex-disjoint cycles in digraphs}, SIAM J. Discrete Math. 33 (2019), 2444-2451.
\bibitem{BBT} J. Bang-Jensen, S. Bessy, S. Thomass\'{e}, \emph{Disjoint 3-cycles in tournaments: A proof of the Bermond-Thomassen conjecture for tournaments}, J. Graph Theory. 75 (2014), 284-302.
\bibitem{BG} J. Bang-Jensen, G. Gutin, \emph{Digraphs: Theory, algorithms and applications}, 2nd ed., Springer-Verlag, London, 2009.
\bibitem{BJHA} J. Bensmail, A. Harutyunyan, N.-K. Le, B. Li, N. Lichiardopol, \emph{Disjoint cycles of different lengths in graphs and digraphs}, Electron. J. Combin. 24 (2017), 1-24.
\bibitem{BT} J.C. Bermond, C. Thomassen, \emph{Cycles in digraphs-A survey}, J. Graph Theory. 5 (1981), 1-43.
\bibitem{BM2008}
A. Bondy, M.R. Murty,
Graph Theory,
Springer-Verlag London, 2008.
\bibitem{B} M. Buci\'{c}, \emph{An improved bound for disjoint directed cycles}, Discrete Math. 341 (2018), 2231-2236.



\bibitem{CC} 
B. Chen, A. Chang, 
\emph{Disjoint cycles in tournaments and bipartite tournaments}, 
J. Graph Theory. 105 (2) (2024), 297-314.
\bibitem{HY} 
M. A. Henning, A. Yeo, 
\emph{Vertex disjoint cycles of different length in digraphs}, 
SIAM J. Discrete Math. 26 (2012), 687-694.

\bibitem{LPS} N. Lichiardopol, A. P\'{o}r, J.-S. Sereni, \emph{A step toward the Bermond–Thomassen conjecture about disjoint cycles in digraphs}, SIAM J. Discrete Math. 23 (2009), 979-992.

\bibitem{L} N. Lichiardopol, \emph{Proof of a conjecture of Henning and Yeo on vertex-disjoint directed cycles}, SIAM J. Discrete Math. 28 (2014), 1618-1627.

\bibitem{M} J.W. Moon, \emph{On subtournaments of a tournament}, Can. Math. Bull. 9 (1966), 297-301.
\bibitem{S} 
R. M. Steiner, 
\emph{Disjoint cycles with length constraints in digraphs of large connectivity or large minimum degree}, 
SIAM J. Discrete Math. 36 (2022), 1343-1362.


\bibitem{T2} 
N.D. Tan, 
\emph{On vertex disjoint cycles of different lengths in 3-regular digraphs}, Discrete Math. 338 (2015), 2485-2491.

\bibitem{T3} 
N.D. Tan, 
\emph{On 3-regular digraphs without vertex disjoint cycles of different lengths}, Discrete Math. 340 (2017), 1933-1943.

\bibitem{T5} 
N.D. Tan, 
\emph{Tournaments and bipartite tournaments without vertex disjoint cycles of different lengths}, SIAM J. Discrete Math. 35 (2021), 485-494.

\bibitem{T} 
C. Thomassen, 
\emph{Disjoint cycles in digraphs}, 
Combinatorica. 3 (1983), 393-396.


\bibitem{WQY} 
Z. Wang, Y. Qi, J. Yan, 
\emph{Vertex-disjoint cycles of the same length in tournaments}, J. Graph Theory. 102 (2023), 666-683.

\bibitem{YBWW} 
D. Yang, Y. Bai, G. Wang, J. Wu, 
\emph{On splitting digraphs}, 
European J. Combin. 71 (2018), 174-179.
\end{thebibliography}
\end{document}